\documentclass{amsart}
\usepackage{amsmath}
\usepackage{amssymb}
\usepackage{amsthm}
\usepackage{amscd}
\usepackage{enumerate} 
\usepackage{amssymb} 
\usepackage{mathrsfs}     
\usepackage[all]{xy}
\usepackage{color}
\usepackage{graphicx}
\usepackage{mathtools}
\usepackage[dvipdfmx]{hyperref}
\usepackage[foot]{amsaddr}
\hypersetup{colorlinks=true}

\title{On Kawamata--Viehweg type vanishing for three dimensional Mori fiber spaces in positive characteristic}
\author{Tatsuro Kawakami}
\address{Graduate School of Mathematical Sciences, University of Tokyo, 3-8-1 Komaba, Meguro-ku, Tokyo 153-8914, Japan}
\email{kawakami@ms.u-tokyo.ac.jp}

\def\phi{\varphi}
\def\epsilon{\varepsilon}
\def\tilde{\widetilde}

\def\mapsto{\longmapsto}

\def\Spec{\operatorname{Spec}}
\def\Proj{\operatorname{Proj}}

\def\codim{\operatorname{codim}}
\def\Pic{\operatorname{Pic}}

\def\Im{\operatorname{Im}}
\def\Ker{\operatorname{Ker}}
\def\Coker{\operatorname{Coker}}

\def\rank{\operatorname{rank}}

\def\sg{\operatorname{sg}}
\def\reg{\operatorname{reg}}

\def\max{\operatorname{max}}

\def\NS{\operatorname{NS}}

\newcommand{\Q}{\mathbb{Q}}

\newcommand{\Z}{\mathbb{Z}}
\newcommand{\PP}{\mathbb{P}}

\newcommand{\sO}{\mathcal{O}}

\theoremstyle{plain}
\newtheorem{thm}{Theorem}[section] 
\newtheorem{cor}[thm]{Corollary}
\newtheorem{prop}[thm]{Proposition}

\newtheorem{lem}[thm]{Lemma}
\theoremstyle{definition} 
\newtheorem{defn}[thm]{Definition}

\newtheorem{eg}[thm]{Example} 

\theoremstyle{remark}
\newtheorem{rem}[thm]{Remark}

\newtheorem*{notation}{Notation} 
\newtheorem*{cl}{Claim}

\newtheorem*{acknowledgement}{Acknowledgments}


\begin{document}
\tolerance = 9999

\maketitle
\markboth{TATSURO KAWAKAMI}{On Kawamata--Viehweg type vanishing in positive characteristic}

\begin{abstract}
In this paper, we prove a Kawamata--Viehweg type vanishing theorem for smooth Fano threefolds, canonical del Pezzo surfaces, and del Pezzo fibrations in positive characteristic.
\end{abstract}

\maketitle
\section{Introduction}
The Kawamata--Viehweg vanishing theorem is one of the most important tools in birational geometry in characteristic zero.
However, this vanishing theorem fails in positive characteristic and a lot of counterexamples
have been constructed (see, for example, \cite{Ray78}, \cite{Muk13}, \cite{CT18}, \cite{CT19}, \cite{Tot19}).
In this paper, we prove that a Kawamata--Viehweg type vanishing theorem still holds on some Mori fiber spaces in positve characteristic.

First, we discuss a Kawamata--Viehweg type vanishing theorem for a smooth Fano threefold.
Indeed, we prove the following theorem.
\begin{thm}[Theorem \ref{BVonfano} and Corollary \ref{KVVonfano}]\label{KVVonfanointro}
Let $X$ be a smooth Fano threefold over an algebraically closed field $k$ of characteristic $p>0$ and $D$ be a Cartier divisor on $X$. If $\sO_X(D)\subset\Omega_X$, then we have $\kappa(X, D)=-\infty$, where $\kappa(X, D)$ is the Iitaka dimension.
Furthermore, if $D$ is nef and $\nu(X, D)>1$, then we have $H^1(X, \sO_X(-D))=0$, where $\nu(X, D)$ is the numerical dimension. 
\end{thm}

Theorem \ref{KVVonfanointro} can be reduced to the case of the Picard rank $\rho(X)=1$ 
by choosing suitable extremal contractions.
The assertion then follows from an application of a result of Shepherd-Barron \cite[Theorem 2.1]{SB18}.
We remark that Theorem \ref{KVVonfanointro} corrects \cite[Theorem 1.4]{SB97} whose proof has a gap in the case where $p=2, 3$.

Next, we consider a Kawamata--Viehweg type vanishing theorem for a del Pezzo surface. 
Cascini--Tanaka \cite[Theorem 4.2 (6)]{CT19} constructed a canonical del Pezzo surface violating the Kawamata--Viehweg vanishing theorem in characteristic two.
Shortly afterward, inspired by their example, Bernasconi \cite[Theorem 1.1]{Ber} constructed a klt del Pezzo surface violating the Kawamata--Viehweg vanishing theorem in characteristic three.
Since his example is not canonical, the Kawamata--Viehweg vanishing theorem for canonical del Pezzo surfaces still has remained open when characteristic is bigger than or equal to three.
In this paper, we give an affirmative answer to this open problem.
More generally, we prove the following theorem.

\begin{thm}[Theorem \ref{KVVonCS}]
Let $X$ be a normal projective surface with canonical singularities over an algebraically closed field of characteristic $p\neq2$. 
Suppose that $\kappa(X)\coloneqq\kappa(\tilde{X}, K_{\tilde{X}})=-\infty$, where $\tilde{X}\to X$ is a resolution.
Then $H^1(X, \sO_X(-D))=0$ for every nef and big $\Q$-Cartier Weil divisor $D$ on $X$. 
\end{thm}

Finally, we focus on a Kawamata--Viehweg type vanishing theorem for a del Pezzo fibration $f \colon Y \to Z$.
Patakfalvi--Waldron \cite[Theorem 1.10]{PW17} proved that $H^1(Y, \sO_Y(-D))=0$ for an ample Cartier divisor $D$ on $Y$ when $Y$ is smooth.
We generalize their result to the case where $Y$ is $\Q$-factorial and has only isolated singularities and $D$ is a nef $\Q$-Cartier Weil divisor such that the numerical dimension $\nu(X, D)$ is bigger than one.
In \cite[Theorem 1.10]{PW17}, they also showed that $H^2(Y, \sO_Y(-D))$ is related to cohomologies of non-normal fibers of $f$. Making use of their result, we deduce the vanishing of $H^2(Y, \sO_Y(-D))$ from an analysis of non-normal fibers and give an answer to \cite[Remark 1.11]{PW17}.
\begin{thm}[Corollary \ref{KVVondelPezzofib} and Theorem \ref{KVonMFS2}]
Let $f \colon Y \to Z$ be a Mori fiber space from a normal projective threefold with only isolated singularities to a smooth curve over an algebraically closed field $k$ of characteristic $p>0$ and $D$ be a nef $\Q$-Cartier Weil divisor  on $Y$ with $\nu(X, D)>1$.
Then $H^1(Y, \sO_Y(-D))=0$.
Furthermore, if $p\geq11$, $Y$ is smooth and $D$ is ample, then $H^2(Y, \sO_Y(-D))=0$.
\end{thm}

\begin{notation}
Throughout this paper, we work over an algebraically closed field $k$ of characteristic $p>0$. We say $X$ is a \textit{variety} if $X$ is an integral separated scheme of finite type over $k$.
A \textit{curve} (resp.~\textit{surface}, resp.~\textit{threefold}) is a variety of dimension one (resp.~two, resp.~three).
Given a variety $X$, we denote by $X_{\reg}$ (resp.~$X_{\sg}$) the \textit{regular} (resp.~\textit{singular}) \textit{locus} of $X$.
Given a projective variety $X$ and a coherent sheaf $\mathcal{F}$ on $X$, $h^0(X, \mathcal{F})$ denotes $\dim_{k}H^0(X, \mathcal{F})$.
Given a normal variety $X$, $\Omega_X^{[i]}$ denotes $(\Omega^i_X)^{**}$ for all $i\geq0$, where $(-)^{**}$ is the reflexive hull.
Given a normal (resp.~Gorenstein) projective variety $X$, we denote by $K_X$ a corresponding Weil divisor (resp.~Cartier divisor) to a \textit{dualizing sheaf} $\omega_X$.
Given a variety $X$, we denote by $\Pic(X)$ the \textit{Picard group} of $X$ and $\Pic(X)_{\Q}\coloneqq \Pic(X)\otimes_{\Z}\Q$.
Given a projective morphism $f\colon X\to Y$ of normal varieties and two $\Q$-Cartier $\Q$-divisors $D_1, D_2$ on $X$, we say $D_1$ and $D_2$ are \textit{numerically equivalent over} $Y$ and denote $D_1\equiv_{f}D_2$ if $D_1-D_2$ is numerically equivalent to zero on any fiber of $f$. 
We denote by $N^1(X/Y)$ the quotient of $\Pic(X)$ by its subgroup consisting of all isomorphism classes numerically equivalent to zero over $Y$. 
We denote $N^1(X/Y)_{\Q}\coloneqq N^1(X/Y)\otimes_{\Z}\Q$.
We denote by $\rho(X/Y)$ the \textit{relative Picard number}. We denote $N^1(X/\Spec\, k)$ (resp.~$\rho(X/\Spec\, k)$) by $N^1(X)$ (resp.~$\rho(X)$).
Given a projective variety $X$,
$\NS(X)$ denotes the quotient group of $\Pic(X)$ by its subgroup consisting of all isomorphism classes algebraically equivalent to zero.
We refer to \cite[Section 2.3]{KM98} for the definitions of singularities appearing in the minimal model program.
\end{notation}

\section{Preliminaries}
In this section, we summarize lemmas we need the rest of this paper.

\begin{defn}
Let $X$ be a normal projective variety and $D$ be a Cartier divisor on $X$.  
We define the \emph{Iitaka dimension}$\in\{0,1,\cdots,\dim X\}$ as follows.
If $h^0(X, \sO_X(mD)) = 0$ for all $m \in \Z_{>0}$, we say $D$ has Iitaka dimension 
$\kappa(X, D) \coloneqq -\infty$.  Otherwise, set
\[ M \coloneqq \bigl\{ m\in \Z_{>0} \;\big|\; h^0(X, \sO_X(mD)) > 0 \bigr\}, \]
and consider the natural rational mappings
\[ \phi_m : X \dasharrow \mathbb P\bigl(H^0(X, \sO_X(mD))^*\bigr) \quad \text{ for each } m \in M. \]
The Iitaka dimension of $D$ is then defined as
\[ \kappa(X, D) \coloneqq \max_{m \in M} \bigl\{ \dim \overline{\phi_m(X)} \bigr\}. \]
We say $D$ is \emph{big} if $\kappa(X, D) = \dim X$.
Note that $\kappa(X, D)=\kappa(X, nD)$ for every $n\geq 1.$ Thus we define $\kappa(X, D)$ for a $\Q$-Cartier $\Q$-divisor $D$ as $\kappa(X, nD)$,
where $n$ is any positive integer such that $nD$ is Cartier.
If a resolution $\tilde{X}\to X$ exists, then we denote $\kappa(\tilde{X}, K_{\tilde{X}})$ by $\kappa(X)$. 
\end{defn}

\begin{defn}
Let $X$ be a normal projective variety and $D$ be a nef $\Q$-Cartier $\Q$-divisor on $X$.
We define the \emph{numerical dimension}$\in\{0,1,\cdots,\dim X\}$ as follows.
If $D$ is numerically trivial, then 
we set $\nu(X, D)=0$. 
If $D$ is not numerically trivial, then
\[
\nu(X, D)\coloneqq\max\{m\in \Z_{>0}
\,|\, D^m \,\,\,
{\rm is\,\,\,not\,\,\,numerically\,\,\,trivial}\}.
\]
Note that $D$ is numerically trivial if 
and only if $\nu(X, D)=0$, and $D$ is big if and only if $\nu(X, D)=\dim\, X$. We always have the inequality $\kappa(X, D)\leq\nu(X, D)$.
\end{defn}

Let $f \colon X \dasharrow Y$ be a birational map of normal varieties. We say $f$ is a {\em birational contraction} if $f^{-1}$ does not contract any divisor. Any birational map which appears in sequences of MMPs is a birational contraction.
\begin{lem}[\textup{\cite[Lemma 2.2]{Kaw19}}]\label{push}
Let $f \colon X\dasharrow X'$ be a birational contraction of normal $\Q$-factorial projective varieties and $D$ be a $\Q$-Cartier Weil divisor on $X$.
Let $D'$ be the pushforward of $D$ by $f$. Then
$H^0(X, (\Omega^{[i]}_X\otimes \sO_X(-D))^{**})\subset H^0(X', (\Omega^{[i]}_{X'}\otimes \sO_{X'}(-D'))^{**})$ for all $i\geq0$ and $\kappa(X, D)\leq \kappa(X', D')$.
\end{lem}

\begin{lem}[{\cite[Theorem 2.11]{Tan15}}]\label{KVVonsur}
Let $X$ be a normal projective klt surface and $D$ be a nef and big $\Q$-Cartier Weil divisor.
Let $H$ be a nef and big Cartier divisor on $X$. 
Then there exits $m_0>0$ such that $H^1(X, \sO_X(-D-mH-N))=0$ for $m\geq m_0$ and for any nef Cartier divisor $N$ on $X$.
\end{lem}
\begin{proof}
This is an immediate consequence of \cite[Theorem 2.11]{Tan15} and the Serre duality for Cohen--Macaulay sheaves (see \cite[Theorem 5.71]{KM98}).
\end{proof}

\begin{lem}\label{BVtoKV}
Let $X$ be a normal projective variety and $D$ be a nef $\Q$-Cartier Weil divisor with 
$\nu(X, D)>1$.
Suppose that one of the following conditions is satisfied.
\begin{enumerate}\renewcommand{\labelenumi}{$($\textup{\arabic{enumi}}$)$}
\item{$X$ is a klt surface.}
\item{$X$ is a threefold with only isolated singularities.}
\end{enumerate}
If $H^0(X, (\Omega^{[1]}_X\otimes\sO_X(-p^eD))^{**})=0$ for all $e>0$, then $H^1(X, \sO_X(-D))=0$.
\end{lem} 
\begin{proof}
We have the exact sequence
\[
0 \to \sO_{X_{\reg}} \to F_{*}\sO_{X_{\reg}} \to F_{*}\Omega^1_{X_{\reg}}.
\]
By tensoring $\sO_{X_{\reg}}(-D)$ and taking the pushforward by the inclusion map $i \colon X_{\reg}\hookrightarrow X$, we get
\[
0 \to \sO_X(-D) \to F_{*}(\sO_X(-pD)) \to F_{*}(\Omega^{[1]}_X\otimes \sO_X(-pD))^{**}.
\]
Let $\mathcal{C}$ denote $\Coker(\sO_X(-D) \to F_{*}(\sO_X(-pD)))\subset F_{*}(\Omega^{[1]}_X\otimes \sO_X(-pD))^{**}$.
By assumption, we have
$H^0(X, \mathcal{C})\hookrightarrow H^0(X, (\Omega^{[1]}_X\otimes \sO_X(-pD))^{**})=0$.
Therefore we get an injective map
$H^1(X, \sO_X(-D))\hookrightarrow H^1(X, \sO_X(-pD))$. By repeating this, it suffices to show $H^1(X, \sO_X(-p^eD))=0$ for sufficiently large $e$. 

First, we assume $X$ is a klt surface and prove (1).
We fix $m$, $n \in \Z_{>0}$ such that $D'\coloneqq p^m(p^n-1)D$ is Cartier.
Then, for any $l \in \Z_{>0}$, we have
\[
\begin{array}{rl}
H^1(X, \sO_X(-p^{m+ln}D))=&H^1(X, \sO_X(-p^mD)\otimes \sO_X(-p^m(p^{ln}-1))D))\\
                                        =&H^1(X, \sO_X(-p^mD)\otimes \sO_X(-H_l)),\\
\end{array}
\]
where $H_l$ denotes $(1+p^n+\cdots +p^{(l-1)n})D'$.
Then, by Lemma \ref{KVVonsur}, we have $H^1(X, \sO_X(-p^mD)\otimes \sO_X(-H_l))=0$ for sufficiently large $l\gg0$. 

Next, we assume $X$ is a threefold with only isolated singularities and prove (2).
We show $H^1(X, \sO_X(-mD))=0$ for all sufficiently large $m\gg0$.
We can take a very ample divisor $H$ on $X$ such that $H$ is a smooth projective surface and $D|_H$ is a nef and big Cartier divisor. 
By Lemma \ref{KVVonsur}, there exists $m_0>0$ such that $H^1(H, \sO_H(-mD-nH))=0$ for all $m\geq m_0$ and $n\geq0$.
Then, by the exact sequence
\[
H^1(X, \sO_X(-mD-(n+1)H))\to H^1(X, \sO_X(-mD-nH))\to H^1(H, \sO_H(-mD-nH)),
\]
we can reduce the vanishing of $H^1(X, \sO_X(-mD))$ to that of $H^1(X, \sO_X(-mD-nH))$ for sufficiently large $n\gg0$.
Thus we may assume $X\subset \PP_{k}^N$ and $\sO_X(H)=\sO_{\PP_k^N}(1)|_{X}$ for some $N>0$.
By the Serre duality and \cite[III, Proposition 6.9]{Har}, we have 

\begin{align*}
H^1(X, \sO_X(-mD-nH))=&H^1(\PP_k^N, \sO_X(-mD)\otimes \sO_{\PP_k^N}(-n))\\
\simeq&\mathrm{Ext}^{N-1}(\sO_X(-mD), \omega_{\PP_{k}^N}(n))& \\
                   \simeq&H^0(\PP_{k}^N, \mathcal{E}xt^{N-1}_{\PP_{k}^N}(\sO_X(-mD), \omega_{\PP_{k}^N}(n)))\\
                   \end{align*}
for $n\gg0$.
Now, since $\sO_X(-mD)$ satisfies the Serre condition $S_2$, 
we have 
\[
\mathrm{pd}_{\sO_{\PP_{k}^N,x}}(\sO_{X}(-mD))_{x}=N-\mathrm{depth}_{\sO_{\PP_{k}^N,x}}(\sO_{X}(-mD))_{x}=N-2
\]
for each closed point $x\in X$,
where $\mathrm{pd}_{\sO_{\PP_{k}^N,x}}(\sO_{X}(-mD))_{x}$ denotes the projective dimension of $\sO_{X}(-mD)_{x}$. 
Therefore $\mathcal{E}xt^{N-1}_{\PP_{k}^N}(\sO_X(-mD), \omega_{\PP_{k}^N}(n)))=0$
and we obtain the assertion.

\end{proof}

Separably uniruled (resp.~separably rationally connected) varieties are analogs of uniruled (resp.~rationally connected) varieties in characteristic zero. If $X$ is a smooth projective surface, then $X$ is separably uniruled (resp. separably rationally connected) if and only if $\kappa(X)=-\infty$ (resp.~$X$ is rational).
We refer to \cite[Chapter IV]{Kol96} for more details.
\begin{lem}[\textup{\cite[Lemma 7]{Kol95} and \cite[Proposition 3.4]{Kaw19}}]\label{BVonSRC}
Let $X$ be a smooth projective variety and $D$ be a Cartier divisor on $X$.
Then the following hold.
\begin{enumerate}\renewcommand{\labelenumi}{$($\textup{\arabic{enumi}}$)$}
\item{If $X$ is separably uniruled and $D$ is big, then $H^0(X, \Omega_X^i \otimes \sO_X(-D))=0$ for all $i\geq0$.}
\item{If $X$ is separably rationally connected and $\kappa(X, D)\geq 0$, then $H^0(X, \Omega_X^i \otimes \sO_X(-D))=0$ for all $i>0$.}
\end{enumerate}
\end{lem}

\begin{lem}\label{Cartier Vanishing}
Let $X$ be a normal projective surface such that $\kappa(X)=-\infty$ and $D$ be a nef and big Cartier divisor.
Then $H^i(X, \sO_X(-D))=0$ for $i<2$.
\end{lem}
\begin{proof}
In the case where $i=0$, the assertion follows immediately from the bigness of $D$.
Thus we assume $i=1$.
We take a resolution $\pi\colon \tilde{X}\to X$.
Then by the spectral sequence 
\[
E_2^{p,q}=H^{p}(X,R^{q}\pi_{*}\sO_{\tilde{X}}(-\pi^{*}D))\Rightarrow E^{p+q}=H^{p+q}(\tilde{X},\sO_{\tilde{X}}(-\pi^{*}D)),
\]
we obtain an injective map $H^{1}(X,\sO_{X}(-D))\hookrightarrow H^1(\tilde{X}, \sO_{\tilde{X}}(-\pi^{*}D))$.
Then it follows from Lemmas \ref{BVtoKV} (1) and \ref{BVonSRC} (1) that $H^1(\tilde{X}, \sO_{\tilde{X}}(-\pi^{*}D))=0$.
\end{proof}
\section{Kawamata--Viehweg type vanishing for smooth Fano threefolds.}

In this section, we prove the cotangent bundle $\Omega_X$ of a smooth Fano threefold $X$ does not contain a line bundle $\sO_X(D)$ with the Iitaka dimension $\kappa(X, D)\geq0$. 
For this, we reduce to the case where the Picard rank is equal to one by choosing suitable extremal contractions.
As a corollary, we obtain the vanishing of $H^1(X, \sO_X(-D))$ for every nef Cartier divisor $D$ with $\nu(X, D)>1$.

\begin{defn}[\textup{\cite[Definition 1.3]{MM83}}]
Let $X$ be a smooth Fano threefold.
We say $X$ is \emph{primitive} if $X$ is not isomorphic to a blowing-up of a smooth Fano threefold along a smooth curve.
\end{defn}

\begin{lem}\label{fanoMMP}
Let $X$ be a smooth Fano threefold.
Then there is a sequence of birational morphisms of smooth Fano threefolds,
\[
X=:X_0 \overset{\varphi_0}{\rightarrow} X_1 \overset{\varphi_1}{\rightarrow} \cdots \overset{\varphi_{\ell-1}}{\rightarrow} X_{\ell} \overset{f}{\rightarrow} Y 
\]
such that the following properties hold. 
\begin{enumerate}\renewcommand{\labelenumi}{$($\textup{\arabic{enumi}}$)$}
\item 
For any $i \in \{0, \ldots, \ell\}$, 
$X_i$ is a smooth Fano threefold.
\item 
For any $i \in \{0, \ldots, \ell-1\}$, 
$\varphi_i:X_i \rightarrow X_{i+1}$ is a blowing up along a smooth curve. 
\item $f\colon X_{l}\to Y$ is a Mori fiber space (see Definition \ref{MFS} for the definition), where $Y$ is a point, $\PP_k^2$, or $\PP^1_k\times \PP^1_k$. A general fiber of $f$ is $\PP_k^1$ when $Y$ is $\PP_k^2$ or $\PP^1_k\times \PP^1_k$. 
\end{enumerate}
\end{lem}
\begin{proof}
We may assume $X$ is primitive.
Then, by \cite[section 8 (8.1),(8.2)]{MM83}, it follows that $\rho(X)=1$ or there exists an extremal ray contraction $f \colon X \to Y$, where $Y$ is a smooth projective surface.
We remark that the argument of \cite[section 8 (8.1),(8.2)]{MM83} works in all characteristic
since an extremal contraction theorem on smooth projective threefolds has been proved in all characteristic in \cite[(1.1) Main Theorem]{Kol91}.
If a general fiber of $f$ is not smooth, then $X$ is isomorphic to the variety of \cite[Corollary 8 (1)]{MS03} and $X$ has another extremal contraction $f' \colon X \to Y'$ which gives a $\PP_k^1$-bundle structure by \cite[Remark 9]{MS03}. By replacing $f$ with $f'$, we may assume a general fiber of $f$ is smooth. 
Note that $X$ can not be isomorphic to the variety of \cite[Corollary 8 (2)]{MS03} since this is not primitive by \cite[Remark 10]{MS03}. 

Next, we show $Y\simeq\PP_k^2$ or $ \PP^1_k\times \PP^1_k$ when $\dim \,Y=2$.
Since $X$ is rationally chain connected by \cite[V, 2.13 Theorem]{Kol96}, so is $Y$. Also, an application of \cite[Lemma 2.4]{Sai03} gives $\kappa(Y)=-\infty$. Therefore $Y$ is a smooth rational surface.
Then it suffices to show there does not exist a curve whose self-intersection is negative. This follows from an argument similar to \cite[Proposition 6.6]{MM83}.
We remark that \cite[Proposition 4.5]{MM83} which is used in the proof of 
\cite[Proposition 6.6]{MM83} is correct in all characteristic. We refer to \cite[Proposition 2.3]{MM86} for the proof which works in all characteristic.
\end{proof}

\begin{thm}[\textup{\cite[Theorem 2.1]{SB18}}, cf.~{\cite[Theorem 1.4]{SB97}}]\label{fano}
Let $X$ be a smooth Fano threefold of $\rho(X)=1$.
Then $H^1(X, \sO_X(-A))=0$ for every ample Cartier divisor $A$ on $X$.
\end{thm}
\begin{rem}
Theorem \ref{fano} was originally claimed in \cite[Theorem 1.4]{SB97}, but the proof has a gap in the case where $p=2$ or $3$. 
Shepherd-Barron corrected his proof in \cite[Theorem 2.1]{SB18}, which is included as an appendix of this paper.
\end{rem}

Here, let us recall the Cartier operators.
Let $X$ be a smooth variety.
The Frobenius pushforward of the de Rham complex
\[
F_{*}\Omega^{\bullet}_X : F_{*}\sO_X \overset{F_{*}d}{\to} F_{*}\Omega_X \overset{F_{*}d}{\to} \cdots
\]
is a complex of $\sO_X$-module homomorphisms.
We define coherent $\sO_X$-modules as follows.
\[
\begin{array}{rl}
&B_X\coloneqq\Im(F_{*}d : F_{*}\sO_X \to F_{*}\Omega_X),\\
&Z_X\coloneqq\Ker(F_{*}d : F_{*}\Omega_X \to F_{*}\Omega^{2}_X).\\
\end{array}
\]
Then we have the exact sequence
\begin{align}
0 \to \sO_X \to F_{*}\sO_X \overset{F_{*}d}{\to} B_X\to 0. \tag{3.1}
\end{align}
Also, we have the exact sequence arising from the Cartier isomorphism,
\begin{align}
0 \to B_X \to  Z_X \overset{C}{\to} \Omega_X \to 0. \tag{3.2}
\end{align}
We refer to \cite[1.3.4 Theorem]{fbook} for the details.

Next, let us recall the Riemann-Roch theorem for smooth projective threefolds.
Let $X$ be a smooth projective threefold and $D$ be a Cartier divisor on $X$.
Then the Riemann-Roch theorem states
\begin{align}
\mathcal{X}(\sO_X(D))&=\frac{1}{12}D\cdot (D-K_X)\cdot (2D-K_X)+\frac{1}{12}D \cdot c_2(X)+\mathcal{X}(\sO_X)\tag{3.3},\\
\mathcal{X}(\sO_X)&=\frac{1}{24}(-K_X) \cdot c_2(X) \tag{3.4}.
\end{align}
We refer to \cite[A. Exercise 6.8]{Har} for the details.

\begin{thm}\label{BVonfano}
Let $X$ be a smooth Fano threefold.
Then $H^0(X, \Omega_X\otimes \sO_X(-D))=0$ for every Cartier divisor $D$ with $\kappa(X, D)\geq 0$. 
\end{thm}
\begin{proof}
We consider the sequence of Theorem \ref{fanoMMP} and use the same notation.
By Lemma \ref{push}, we may assume $X$ is primitive.
If $\dim\,Y\neq0$, then $X$ is separably rationally connected by \cite[Theorem 0.5]{GLP$^+$15} and the assertion follows from Lemma \ref{BVonSRC} (2). Thus we may assume  $\rho(X)=1$.
Then an application of Theorem \ref{fano} gives $H^1(X, \sO_X)=H^2(X, \sO_X)=0$ and $\Pic(X)\simeq \Z$.
This follows from the same argument as in \cite[Corollary 1.5 (1), (2)]{SB97}, but we prove here for the convenience of the reader. 
Since $X$ is rationally chain connected by 
\cite[V, 2.13 Theorem]{Kol96}, we have $\Pic^0(X)=0$ and hence $\Pic(X)=\NS(X)$.
By \cite[Theorem 4.2]{XZ19}, it follows from $\Pic^0(X)=0$ that $h^1(X, \sO_X)\leq h^2(X, \sO_X)$.
Since $h^2(X, \sO_X)=h^1(X, \sO_X(K_X))=0$ by Theorem \ref{fano}, we obtain $h^1(X, \sO_X)=0$ and $\mathcal{X}(\sO_X)=1$.
Let $\sO_X(B)\in \Pic(X)$ be a torsion element. Since $h^2(X, \sO_X(B))=h^1(X, \sO_X(K_X-B))=0$ by Theorem \ref{fano}, we get $h^0(X, \sO_X(B))\geq\mathcal{X}(\sO_X(B))\underset{B\equiv0}{=}\mathcal{X}(\sO_X)=1$. Thus $\sO_X(B)\simeq \sO_X$ and hence $\Pic(X)$ is torsion-free. Together with $\Pic(X)=\NS(X)$, we obtain $\Pic(X)= N^1(X)\simeq\Z$.

Now, let us show $H^0(X, \Omega_X\otimes \sO_X(-D))=0$, where $D$ is a Cartier divisor with $\kappa(X, D)\geq 0.$
We first consider the case where $\kappa(X, D)=0$. In this case, it follows from $\Pic(X)\simeq \Z$ that $D=0$, and thus it suffices to show $H^0(X, \Omega_X)=0$.
By the exact sequence (3.1) and $H^1(X, \sO_X)=H^2(X, \sO_X)=0$, we get $H^0(X, B_X)=H^1(X, B_X)=0$.
Then, by the exact sequence (3.2), we obtain $H^0(X, Z_X)\overset{C}{\simeq} H^0(X, \Omega_X)$ and
the natural inclusion map $H^0(X, Z_X)\hookrightarrow H^0(X, \Omega_X)$ is isomorphic.
Now, an application of \cite[Proposition 4.3]{GK03} gives
\[
H^0(X, \Omega_X)\simeq \Pic(X)[p]\otimes_{\mathbb{F}_p} k,
\]
where $\Pic(X)[p]$ denotes a subgroup of $\Pic(X)$ composed by $p$-torsion elements.
Since $\Pic(X)$ is torsion-free, we obtain $H^0(X, \Omega_X)=0$.
Next, we discuss the case where $\kappa(X, D)>0$.
Since $\rho(X)=1$, we can denote $D\equiv a(-K_X)$ for some $a\in \Q_{>0}$. 
Together with $\mathcal{X}(\sO_X)=1$, the Riemann-Roch theorem gives
\[
\begin{array}{rl}
\mathcal{X}(\sO_X(D))\underset{(3.3)}{=}&\frac{1}{12}D\cdot (D-K_X)\cdot (2D-K_X)+\frac{1}{12}D \cdot c_2(X)+\mathcal{X}(\sO_X)\\
                                 \underset{(3.4)}{=}&\frac{1}{12}D\cdot (D-K_X)\cdot (2D-K_X)+2a+1\\
                                 >&0.
\end{array}
\]
Now we have $h^0(X, \sO_X(D))\geq \mathcal{X}(\sO_X(D))>0$ since
$h^2(X, \sO_X(D))= h^1(X, \sO_X(K_X-D))=0$ by Theorem \ref{fano}.
Therefore $D$ is linearly equivalent to an effective divisor, and hence we obtain $H^0(X, \Omega_X\otimes \sO_X(-D))\hookrightarrow H^0(X, \Omega_X)=0$.
\end{proof}

\begin{cor}\label{KVVonfano}
Let $X$ be a smooth Fano threefold and $D$ be a nef Cartier divisor with $\nu(X, D)>1$.
Then $H^1(X, \sO_X(-D))=0$. 
\end{cor}
\begin{proof}
If $\nu(X, D)=3$, that is, if $D$ is nef and big, then the assertion follows from Lemma \ref{BVtoKV} (2) and Theorem \ref{BVonfano}.
Now, we assume $\nu(X, D)=2$. Since $\nu(X, mD-K_X)=3$, we have $h^2(X, \sO_X(mD))=h^1(X, \sO_X(K_X-mD))=0$ for all $m>0$ by the argument above. 
Then the Riemann-Roch theorem (3.3) gives
\[
h^0(X, \sO_X(mD))\geq \frac{D^2\cdot (-K_X)}{4}m^2+\frac{D\cdot K_X^2+D\cdot c_2(X)}{12}m+\mathcal{X}(\sO_X).
\]
Since $D^2\cdot (-K_X)>0$, we have $\kappa(X, D)=2$ and hence $H^0(X, \Omega_X\otimes\sO_X(-nD))=0$ for all $n>0$ by Theorem \ref{BVonfano}. Then the assertion follows from Lemma \ref{BVtoKV} (2). 
\end{proof}

\begin{cor}[\textup{cf.~\cite[Corollary 1.5]{SB97}}]\label{app}
Let $X$ be a smooth Fano threefold.
Then the following hold.
\begin{enumerate}\renewcommand{\labelenumi}{$($\textup{\arabic{enumi}}$)$}
\item{$H^i(X, \sO_X)=0$ for all $i>0$. In particular, $\mathcal{X}(\sO_X)=1$.}
\item{$\Pic(X)=\Z^{\rho(X)}$.}
\item{$X$ is simply connected, i.e., any finite \'etale morphism $f\colon Y\to X$ is an isomorphism.}
\end{enumerate}
\end{cor}
\begin{proof}
We refer to \cite[Corollary 1.5]{SB97} for the proof.
\end{proof}
\section{Kawamata--Viehweg type vanishing on canonical del Pezzo surfaces}
In this section, we consider a Kawamata--Viehweg type vanishing theorem on canonical del Pezzo surfaces.
\begin{defn}\label{GdelPezzo}
We say $X$ is a \emph{canonical del Pezzo surface} if $X$ is a normal projective surface with canonical singularities such that $-K_X$ is ample.
\end{defn}

If $p=2$, by taking $d=3, q_1=1, q_2=2$ in \cite[Theorem 4.2 (6)]{CT19}, we obtain a canonical del Pezzo surface $X$ violating the Kawamata-Viehweg vanishing theorem, that is, 
$H^1(X, \sO_X(-A))\neq0$ for some ample $\Q$-Cartier Weil divisor $A$. 
In this section, we prove this pathological phenomenon can not happen unless $p=2$.

First, we gather the basic facts about canonical del Pezzo surfaces in the following lemma. 
\begin{lem}\label{basic}
Let $X$ be a canonical del Pezzo surface.
Then the following hold.
\begin{enumerate}\renewcommand{\labelenumi}{$($\textup{\arabic{enumi}}$)$}
\item{$\dim|-K_X|=K_X^2$.}
\item{$|-K_X|$ has no fixed part.}
\item{A general member of $|-K_X|$ is a locally complete intersection (l.c.i. for short) curve with arithmetic genus one.}
\item{If $K_X^2\geq3$, then $\omega_X^{-1}$ is very ample.}
\item{If $K_X^2=2$, then $\omega_X^{-1}$ is globally generated and $\omega_X^{-2}$ is very ample.}
\item{If $K_X^2=1$, then $\omega_X^{-2}$ is globally generated and $\omega_X^{-3}$ is very ample.} 
\end{enumerate}
\end{lem}
\begin{proof}
We refer to \cite[Proposition 2.10, Proposition 2.12 and Proposition 2.14]{BT} for the proof.
\end{proof}

\begin{rem}\label{sing}
\cite[Proposition 4.2 (ii)]{HW81} claims that a general anti-canonical member of a canonical del Pezzo surface is smooth. However, their proof relies on \cite[III, Th\'eor\`eme 1]{Dem}, which is valid only in characteristic zero because of the Bertini theorem (see a remark in the beginning of \cite[Chapter IV]{Dem}).
Indeed, there exist canonical del Pezzo surfaces whose anti-canonical members are all singular when $p=2$ or $3$ as we see in Example \ref{counter} and Example \ref{counter2}.
We will see in  Proposition \ref{smooth} that it is smooth unless $p=2$ or $3$.
\end{rem}

\begin{eg}\label{counter}
Let $p=2$ or $3$.
We take $C_{0}\coloneqq \{X^3+Y^2Z=0\}$ and $L_{0}\coloneqq \{Z=0\}$ on $\PP_k^2=\Proj k[X,Y,Z]$. Then $C_{0}$ intersects with $L_{0}$ at a single point $[0:1:0]$ with multiplicity three.
Let $V$ be the pencil generated by cubic curves $C_{0}$ and $3L_{0}= \{Z^3=0\}$. 
Then the Jacobian criterion shows that all members of $V$ are singular.

Let us consider blow-ups as follows.
First, we blow up at $C_{0}\cap L_{0}=[0:1:0]$. 
From the second time, we blow up at the intersection of the strict transform $C_0$ and the exceptional divisor of the blow-up just before.
We denote by $h_j$ (resp.~$Y_j$, $F_{i,j}$ ($i\leq j$), $C_j$, $L_j$) the $j$-th blow-up (resp.~the resulting surface of the $j$-th blow-up, the strict transform of the exceptional divisor of the $i$-th blow-up to $Y_j$, the strict transform of $C_{0}$  to $Y_j$, the strict transform of $L_{0}$ to $Y_j$). Since $C_{0}$ is smooth at $C_{0}\cap L_{0}$, it follows that $C_{i}$ and $F_{i,i}$ intersect transversely at a single point $C_{i}\cap F_{i,i}$. Note that $h_{i+1}$ is a blow-up at a point $C_{i}\cap F_{i,i}$ for $i>0$.

Now, we prove that $F_{1,8},F_{2,8},\cdots,F_{7,8}$, and $L_{8}$ are $(-2)$-curves and $\sum_{i=1}^7 F_{i,8}+L_8$ is of $E_8$-type.
Since $C_{i}$ and $F_{i,i}$ intersect transversely at a point $C_{i}\cap F_{i,i}$ for $i>0$, it follows that $\sum_{i=1}^7 F_{i,8}$ is a linear chain of $(-2)$-curves. 
Let us prove that $L_8$ is a $(-2)$-curve such that
$L_8\cdot F_{i,8}=1$ (resp.~$0$) when $i=3$ (resp.~$i\neq3$). 
Since $C_0$ intersects with $L_{0}$ at $[0:1:0]$ with multiplicity three, we have $C_i \cdot L_i=\max\{3-i, 0\}$ for $i\geq 0$.
In particular, a point $C_i \cap F_{i,i}$ is contained in $L_i$ if and only if $i=1$ or $2$.
From this, we can see that $L_8^2=-2$ and $L_8 \cdot F_{i,8}=1$ (resp.~$0$) when $i=3$ (resp.~$i\neq3$).
Hence $F_{1,8},F_{2,8},\cdots,F_{7,8}$, and $L_{8}$ are $(-2)$-curves and $\sum_{i=1}^7 F_{i,8}+L_8$ is of $E_8$-type.

Let $Y\coloneqq Y_8$ and $\pi\colon Y\to X$ be the contraction of these $(-2)$-curves. Then $X$ is a canonical del Pezzo surface with an $E_8$-singularity because $X$ is rational and $\rho(X)=\rho(Y)-8=1$. Moreover, $Y$ is a weak del Pezzo surface since $-K_Y=\pi^{*}(-K_X)$ is nef and big.
Now, for the sake of contradiction, let us assume there exists a smooth member $C\in |-K_X|$.
Since $C\subset X_{\reg}$, it follows that $\pi^{-1}(C)$ is a smooth member of $|-K_Y|$.
Since $\pi^{-1}(C)\cdot F_{8,8}=(-K_Y) \cdot F_{8,8}=1$, it follows that $\pi^{-1}(C)$ and $F_{8,8}$ intersect transversely and thus $(h_8)_{*}(\pi^{-1}(C))$ is smooth. By repeating this, we can show the pushforward of $\pi^{-1}(C)$ to $\PP^2_k$ gives a smooth member of $V$, a contradiction.
\end{eg}

\begin{eg}[\textup{\cite[Proposition 4.3]{CT18}}]\label{counter2}
Let $p=2$ and $V$ be the linear system generated by cubic curves $\{X^2Y+Y^2X=0\}, \{Y^2Z+Z^2Y=0\}$, and $\{Z^2X+X^2Z=0\}$ on $\PP_k^2=\Proj k[X,Y,Z]$. Then the Jacobian criterion shows that all members of $V$ are singular.
Let $h\colon Y\to \PP^2_{k}$ be a blow-up at $[1:0:0], [0:1:0], [0:0:1], [1:1:0], [1:0:1], [0:1:1], [1:1:1]$. Then the strict transforms of lines $\{X=0\}$, $\{Y=0\}$, $\{Z=0\}$, $\{X+Y=0\}$, $\{Y+Z=0\}$, $\{Z+X=0\}$, and $\{X+Y+Z=0\}$ to $Y$ form seven disjoint $(-2)$-curves.
Let $\pi\colon Y\to X$ be the contraction of all $(-2)$-curves. Since $\rho(X)=\rho(Y)-7=1$ and $X$ is rational, it follows that $X$ is a canonical del Pezzo surface with seven $A_1$-singularities. 
Moreover, $Y$ is a weak del Pezzo surface since $-K_Y=\pi^{*}(-K_X)$ is nef and big.
Now, if there exists a smooth member of $|-K_X|$, then this gives a smooth member of $V$ by the same argument as in Example \ref{counter}. This is a contradiction.
\end{eg}

These pathological phenomena are related to quasi-elliptic fibrations, which happen only in $p=2$ or $3$.
In fact, a general anti-canonical member of a canonical del Pezzo surface is smooth if $p\geq 5$ as follows. 

\begin{prop}\label{smooth}
Let X be a canonical del Pezzo surface.
If $p\geq 5$, then a general member of $|-K_X|$ is smooth.
\end{prop}
\begin{proof}
In the case where $K_X^2\geq3$, the assertion holds since $-K_X$ is very ample by Lemma \ref{basic} (4).
Next, we assume $K_X^2=1$.  
In this case, the base locus of $|-K_X|$ is a point $x$, every member is an l.c.i. curve with arithmetic genus one, and every two members intersect transversely at $x$ by Lemma \ref{basic} (1),(2), and (3). 
By blowing up at $x$, we get a resolution $f \colon Y \to \PP_k^1$ of a pencil $\phi_{|-K_X|}\colon X \dasharrow \PP^1_k$. 
Since each member of $|-K_X|$ is smooth at $x$, each fiber of $f$ is isomorphic to its image on $X$.
Let $\pi \colon \tilde{Y}\to Y$ be a resolution and $\tilde{f}\coloneqq\pi\circ f$.
Since a general fiber of $\tilde{f}$ is reduced and irreducible, we have $\tilde{f}_{*}\sO_{\tilde{Y}}\simeq\sO_{\PP^1_k}$. Then a general fiber of $\tilde{f}$ is smooth if $p\geq 5$, and so is a general member of $|-K_X|$. 

Finally, we assume $K_X^2=2$. Then $\phi_{|-K_X|}\colon X \to \PP^2_k$ is a finite morphism of degree two by Lemma \ref{basic} (5), and this is generically \'etale since $p\neq2$.
Let $C_1$ and $C_2$ be the pullbacks of two general members of $|\sO_{\PP_k^2}(1)|$. Then $C_1$ and $C_2$ intersect transversely at two points $x$ and $y$. Let $V$ be a linear system generated by $C_1$ and $C_2$. By blowing up at $x$ and $y$, we get a resolution $f \colon Y \to \PP_k^1$ of a pencil $\phi_{V}\colon X \dasharrow \PP^1_k$.
Since every two members of $V$ intersect transversely at $x$ and $y$, each fiber of $f$ is isomorphic to its image on $X$.
Since a general fiber of $f$ is smooth if $p\geq 5$ as before, so is a general member of $|-K_X|$.
\end{proof}

\begin{rem}
Proposition \ref{smooth} holds when $X$ is a canonical weak del Pezzo surface, that is, a normal projective surface with canonical singularities such that $-K_{X}$ is nef and big.
Indeed, by considering the Stein factorialization of $\phi_{|-mK_X|}$ for sufficiently large $m\gg0$, we obtain $\pi\colon X \to X'$, where $X'$ is a canonical del Pezzo surface and $\pi$ is a contraction of $(-2)$-curves. Since a general member of $|-K_X|$ does not intersect with $(-2)$-curves, this is smooth by Proposition \ref{smooth}.
\end{rem}

Now, we prove a Kawamata--Viehweg type vanishing theorem for normal projective surfaces with  canonical singularities. 
\begin{thm}\label{KVVonCS}
Let $X$ be a normal projective surface with canonical singularities such that $\kappa(X)=-\infty$, where $\tilde{X}\to X$ is a resolution. 
If $p\neq2$,
then $H^1(X, \sO_X(-D))=0$ for every nef and big $\Q$-Cartier Weil divisor $D$ on $X$.
\end{thm}
\begin{rem}
The assumption of characteristic $p$ is sharp by \cite[Theorem 4.2 (6)]{CT19}.
\end{rem}

\begin{proof}
By Lemma \ref{BVtoKV}, it suffices to show $H^0(X, (\Omega^{[1]}_X\otimes\sO_X(-p^mD))^{**})=0$ for all $m>0$. 
Since $X$ has only canonical singularities, the assumption $\kappa(X)=-\infty$ is equivalent to saying that $K_X$ is not pseudo-effective. Then, by running a $K_X$-MMP (see \cite[Theorem 1.1]{Tan12}), we obtain a Mori fiber space.  By Lemma \ref{push}, we may assume $X$ is the Mori fiber space $f \colon X\to Y$. 
For the sake of contradiction, we assume there exists an injective $\sO_X$-module homomorphism $s \colon \sO_X(p^mD) \hookrightarrow \Omega_X^{[1]}$ for some $m>0$.
The case where $\dim\, Y=1$ will be proved in Theorem \ref{BVonMFS} (1). Therefore we assume $\dim Y=0$, that is, $X$ is a canonical del Pezzo surface of $\rho(X)=1$.

We consider the case where $K_X^2=1$.
By Lemma \ref{basic} (6), a general member $C$ of $|-3K_X|$ is a smooth curve.
By restricting $s$ on $C$, we have an injective $\sO_C$-module homomorphism $s|_{C}\colon\sO_{C}(p^mD|_C)\hookrightarrow \Omega^{[1]}_{X}|_{C}$, where
the injectivity follows from the generality of $C$.
The generality of $C$ also shows
$C$ is contained in $X_{\reg}$,  $\sO_{C}(D|_C)$ is ample Cartier, and $\Omega^{[1]}_{X}|_{C}=\Omega^{1}_{X}|_{C}$.
Let $t \colon \sO_C(p^{m}D|_C) \to \omega_C$ be a composition of $s|_C \colon \sO_{C}(p^mD|_C)\hookrightarrow \Omega^{1}_{X}|_{C}$ and a natural map
$\Omega^{1}_{X}|_{C}\to \omega_{C}$.
Then we have the following commutative diagram by the conormal exact sequence.
\begin{equation*}
\xymatrix{ & & \sO_C(p^{m}D|_C) \ar@{.>}[ld] \ar[d]^{s|_C} \ar[rd]^{t} &\\
                 0\ar[r] &\sO_C(-C) \ar[r]   & \Omega^1_X|_{C} \ar[r]  & \omega_C \ar[r] & 0.}
\end{equation*}
Since $p\neq 2$, we have 
\[
\deg_{C}(p^mD|_{C})=C\cdot p^mD=3p^m(-K_X)\cdot D\geq3p^m\geq9,
\]
where the third inequality follows from the fact that $-K_X$ is ample Cartier. On the other hand, we have 
\[
\deg_{C}(\omega_C)=(K_X+C)\cdot C=6K_X^2=6.
\]
Thus $t$ is a zero map and an injective $\sO_C$-module homomorphism $\sO_C(p^mD|_C) \hookrightarrow \sO_C(-C)$ is induced.
This contradicts an anti-ampleness of $\sO_C(-C)=\sO_C(3K_X)$.

If $K_X^2=2$ (resp.~$K_X^2\geq 3$), then by taking a general member of $|-2K_X|$ (resp.~$|-K_X|$), we can derive a contradiction in a similar way.
\end{proof}

Bernasconi \cite[Theorem 1.3]{Ber} proved that $H^1(X, \sO_X(D))=0$ for a klt del Pezzo surface $X$ and a nef and big Cartier divisor $D$ on $X$ when $p\geq 5$.
If we assume $X$ has only canonical singularities, then we can prove the following proposition.
\begin{prop}
Let $X$ be a normal projective surface with canonical singularities such that $-K_{X}$ is nef and big.
Let $D$ be a nef $\Q$-Cartier Weil divisor on $X$ with $\nu(X, D)>0$.
Then $H^1(X, \sO_X(D)))=0$.
\end{prop}
\begin{proof}
First, we show $\kappa(X, D)>0$.
Let $f\colon \tilde{X}\to X$ be the minimal resolution.
Since $f^{*}D$ is nef and $-K_{\tilde{X}}=f^{*}(-K_X)$ is nef and big, it follows that $-K_{\tilde{X}}+f^{*}D$ is nef and big, and in particular $h^2(\tilde{X}, \sO_{\tilde{X}}(f^{*}D))= h^0(\tilde{X}, \sO_{\tilde{X}}(K_{\tilde{X}}-f^{*}D))=0$.
Then the Riemann-Roch theorem on $\tilde{X}$ gives
\[
\begin{array}{rl}
h^0(X, \sO_X(mD))=&h^0(\tilde{X}, \sO_{\tilde{X}}(f^{*}mD))\\
                \geq& \mathcal{X}(\sO_{\tilde{X}}(f^{*}(mD)))\\
                 =&\frac{(f^{*}mD)^2}{2}+\frac{(-K_{\tilde{X}})\cdot f^{*}mD}{2}+\mathcal{X}(\sO_{\tilde{X}})\\
                 =&\frac{D^2}{2}m^2+\frac{(-K_{X})\cdot D}{2}m+\mathcal{X}(\sO_{X})
                 \end{array}
\]
for all $m>0$ such that $mD$ is Cartier. 
Since $-K_X$ is big and $D$ is nef with $\kappa(X, D)>0$, we have $D^2\geq0$, $(-K_{X})\cdot D>0$, and hence $\kappa(X, D)>0$.

Now, let us show $H^1(X, \sO_X(D))=0$.
By the Serre duality, it suffices to show $H^1(X, \sO_X(-(D-K_X)))=0$. Since $D-K_X$ is nef and big, we can reduce to show $H^0(X, (\Omega_X^{[1]}\otimes \sO_X(-p^{m}(D-K_X)))^{**})=0$ for all $m>0$ by Lemma \ref{BVtoKV} (1).
By running a $K_X$-MMP, we obtain a Mori fiber space $f\colon X'\to Y$.  
By Lemma \ref{push}, it suffices to show $H^0(X', (\Omega_{X'}^{[1]}\otimes \sO_{X'}(-p^{m}(D'-K_{X'})))^{**})=0$, where $D'$ is the pushforward of $D$ to $X'$. 
If $\dim \,Y=1$, then the vanishing follows from Theorem \ref{BVonMFS} (1).
Now, we assume $\dim\,Y=0$.
By Lemma \ref{push}, we have $\kappa(X', D')\geq\kappa(X, D)>0$. Since $-K_{X'}$ is ample Cartier, it follows that $D'\cdot (-K_{X'})\geq1$. Then an argument similar to Theorem \ref{KVVonCS} gives the assertion without the assumption of characteristic $p$.
\end{proof}

\section{Kawamata--Viehweg type vanishing for del Pezzo fibrations}
In this section, we discuss a Kawamata--Viehweg type vanishing theorem for del Pezzo fibrations and we generalize the result by Patakfalvi--Waldron \cite[Theorem 1.10]{PW17}.

\begin{defn}\label{MFS}
Let $f \colon X \to Y$ be a projective surjective morphism between normal varieties.
We say $f \colon X \to Y$ is a \emph{Mori fiber space} if
\begin{itemize}
    \item $-K_X$ is $f$-ample,
    \item $X$ is $\Q$-factorial,
    \item $f_{*}\sO_X=\sO_Y$ and $\dim \, X>\dim \, Y$,
    \item the relative Picard rank $\rho(X/Y)=1$.
\end{itemize}
We say $f \colon X \to Y$ is a \emph{del Pezzo fibration} if $f \colon X \to Y$ a Mori fiber space with $\dim X=3$ and $\dim Y=1$.
\end{defn}

Let $f\colon X\to Y$ be a Mori fiber space.
Then $\rho(X/Y)=\rho(X)-\rho(Y)$ is a consequence of the base point free theorem and this theorem needs some assumptions in positive characteristic (see \cite[Theorem 1.3]{HW19} for example). 
On the other hand, when  $f\colon X\to Y$ is a del Pezzo fibration, an application of \cite[Theorem 1.1]{BT} shows $\rho(X/Y)=\rho(X)-\rho(Y)$ as follows.

\begin{lem}\label{rho}
Let $f\colon X\to Y$ be a del Pezzo fibration such that $X$ is projective.
Suppose that one of the following conditions is satisfied.
\begin{enumerate}\renewcommand{\labelenumi}{$($\textup{\arabic{enumi}}$)$}
\item{$X$ has only klt singularities.}
\item{$X$ has only isolated singularities.}
\end{enumerate}
Then $0\to N^1(Y)_{\Q}\overset{f^{*}}{\to} N^1(X)_{\Q} \to N^1(X/Y)_{\Q}\to 0$ is exact. In particular, $\rho(X)=\rho(X/Y)+\rho(Y)=2$.
\end{lem}
\begin{proof}
Let $D$ be a $\Q$-Cartier $\Q$-divisor which is numerically trivial over $Y$. 
Let us show there exists $n\in\Z_{>0}$ such that $nD$ is linearly trivial over $Y$.
In the case (1), this follows immediately from \cite[Theorem 1.1]{BT}.
In the case (2), the same proof works since the generic fiber of $f$ is a regular del Pezzo surface. We refer to the proof of  \cite[Theorem 8.2]{BT} for the details. 
Now, it follows from $f_{*}\sO_X=\sO_Y$ that $\Pic(Y)_{\Q}\overset{f^{*}}{\to}\Pic(X)_{\Q}$ is injective.
Then an argument similar to \cite[Lemma 3-2-5 (2)]{KMM} shows that $\Im(f^{*})=\{\sO_X(D)\in\Pic(X)_{\Q}|D\equiv_{f}0\}$. Therefore we obtain the exact sequence $0\to N^1(Y)_{\Q} \overset{f^{*}}\to N^1(X)_{\Q} \to N^1(X/Y)_{\Q}\to 0$. 
\end{proof}

For a Kawamata--Viehweg type vanishing theorem on a Mori fiber space $f \colon X\to Y$, we focus on the relative Iitaka dimension of a divisorial subsheaf of a reflexive cotangent bundle $\Omega_X^{[1]}$.

\begin{thm}\label{BVonMFS}
Let $f \colon X \to Y$ be a surjective projective morphism between normal varieties with $f_{*}\sO_X=\sO_Y$
and $D$ be a Weil divisor on $X$. 
Let $d\coloneqq\dim\,X-\dim\,Y$.
Suppose that one of the following conditions is satisfied.
\begin{enumerate}\renewcommand{\labelenumi}{$($\textup{\arabic{enumi}}$)$}
\item{$-K_X$ is $f$-ample, $d=1$, and $\dim\,Y\leq1$.}
\item{$-K_X$ is $f$-ample, $d=1$, and $p\neq2$.}
\item{$-K_X$ is $f$-ample, $d=2$, $\rho(X/Y)=1$, $\dim\,Y=1$, and $X$ has only isolated singularities.}
\item{$-K_X$ is $f$-ample, $d=2$, $p\geq5$, and $\codim_X(X_{\sg})\geq3$.}
\end{enumerate}
If there exists an injective $\sO_X$-module homomorphism $\sO_X(D)\hookrightarrow\Omega_X^{[1]}$, then $\kappa(F, D|_F)\leq0$, where $F$ is a general fiber of $f$.
\end{thm}
\begin{proof}
First, we consider the case where (3) and (4).
Since $\codim_X(X_{\sg})\geq 3$, a general fiber $F$ is contained in $X_{\reg}$ and thus $F$ is a normal l.c.i. del Pezzo surface. Here, the normality follows from \cite[Theorem 14.1]{FS18} and \cite[Theorem 1.5]{PW17} for $(3), (4)$, respectively.  
Then $F$ is either a canonical del Pezzo surface or a cone of an elliptic curve by \cite[Theorem 2.2]{HW81}. 
Arguing by contradiction, we assume there exists an injective $\sO_X$-module homomorphism $s \colon \sO_X(D) \hookrightarrow \Omega_X^{[1]}$ for some $D$ with $\kappa(F, D|_F)>0$.
By the generality of $F$, it follows that $s|_{F} \colon \sO_F(D|_F)\hookrightarrow \Omega^{[1]}_X|_{F}=\Omega^{1}_X|_{F}$ is injective and $D|_F$ is a Cartier divisor.
Let $t \colon \sO_F(D|_F) \to \Omega_F$ be a composition of $s|_F \colon \sO_{F}(D|_F)\hookrightarrow \Omega_{X}|_{F}$ and a natural map
$\Omega_{X}|_{F}\to \Omega_{F}$.
Then we have the following diagram by the conormal exact sequence.
\[
\xymatrix{
 &  & \sO_F(D|_F)  \ar@{.>}[ld] \ar@{^{(}->}[d]^{s|_{F}} \ar[rd]^{t} &      & \\
0\ar[r] &         \sO_F^{\oplus\dim\,Y}  \ar[r]^{v} &                    \Omega_X|_F \ar[r] &\Omega_F \ar[r] &0.\\
}
\]
Note that $\mathrm{Ker}(v)=0$ because this is torsion-free and $\rank(\mathrm{Ker}(v))=\rank(\Omega_X|_F)-\rank(\sO_F^{\oplus\dim\,Y})-\rank(\Omega_F)=\dim\,X-\dim\,Y-\dim\,F=0$. 
Let $\pi \colon \tilde{F} \to F$ be a resolution. 
Let us show $H^0(\tilde{F}, \Omega_{\tilde{F}}\otimes \sO_{\tilde{F}}(-\pi^{*}(D|_{F})))=0$. If $F$ is a canonical del Pezzo surface, then $\tilde{F}$ is a smooth rational surface and the vanishing follows from Lemma \ref{BVonSRC} (2). If $F$ is a cone of an elliptic curve, then $\rho(F)=1$ and $\tilde{F}$ is a smooth separably uniruled surface.  Since $D|_F$ is ample by $\rho(F)=1$, it follows that $\pi^{*}(D|_{F})$ is big, and Lemma \ref{BVonSRC} (1) gives the vanishing.
Now, since $F$ is normal and l.c.i., \cite[Theorem 1.1]{Ham06} shows that $\Omega_F$ is torsion-free, and thus the natural map $u \colon \Omega_F \to \pi_{*}\Omega_{\tilde{F}}$ is injective.
Then by 
\[
u\circ t \in H^0(F, \pi_{*} \Omega_{\tilde{F}}\otimes \sO_F(-D|_F))=H^0(\tilde{F}, \Omega_{\tilde{F}}\otimes \sO_{\tilde{F}}(-\pi^{*}(D|_{F})))=0,
\]
it follows that $t$ is a zero map and an injective $\sO_F$-module homomorphism $\sO_F(D|_F)\hookrightarrow \sO_F^{\oplus\dim\,Y}$ is induced. In particular, we obtain an inclusion $\sO_F(D|_{F})\hookrightarrow \sO_F$, a contradiction with $\kappa(F, D|_F)>0$. 

Next, we consider the case where (1) and (2). Since $-K_X$ is $f$-ample, a generic fiber of $f$ is a regular conic. Then a general fiber of $F$ is a smooth conic if $p\neq 2$. If $\dim Y=1$, then the extension of function fields $K(X)/K(Y)$ is separable by \cite[Lemma 7.2]{Bad} and $F$ is reduced, that is, isomorphic to $\PP^1_k$ even if $p=2$. Now, the assertions follow from a similar argument to (3) and (4). 
\end{proof}

\begin{cor}\label{KVVondelPezzofib}
Let $f \colon X \to Y$ be a del Pezzo fibration.
Suppose that $X$ is projective and has only isolated singularities.
Let $D$ be a nef $\Q$-Cartier Weil divisor on $X$ with $\nu(X, D)>1$.
Then $H^1(X, \sO_X(-D))=0$.
\end{cor}
\begin{proof}
By Lemma \ref{rho} (2), we can denote $D\equiv a(-K_X)+bF$, where $F$ is a fiber of $f$.
Since $\nu(X, F)=1$, we get $a>0$ and $D|_F$ is ample. Then the assertion follows from Theorem \ref{BVonMFS} (3) and Lemma \ref{BVtoKV} (2).
\end{proof}

\begin{rem}
Let $X$ be a normal projective threefold with terminal singularities. If $p\geq 5$, then we can take a small $\Q$-factorialization and run a $K_{X}$-MMP
by \cite[Theorem 1.2]{HW19} and the proof of \cite[Theorem 1.7]{Bir}.
Let us assume the output of the MMP is a Mori fiber space $f \colon X' \to Y$ with $\dim \, Y>0$. For example, this happens when $X$ is not rationally chain connected and $K_X$ is not pseudo-effective. Then we obtain $H^1(X, \sO_X(-D))=0$ for every nef and big $\Q$-Cartier Weil divisor $D$ on $X$ by Lemma \ref{push}, Lemma \ref{BVtoKV} (2), and Theorem \ref{BVonMFS} (2), (3). Note that three-dimensional terminal singularities are isolated in all characteristic by \cite[Corollary 2.13]{Kol13}.
\end{rem}

Next, we discuss the vanishing of $H^2(X, \sO_X(-A))$ for a smooth del Pezzo fibration $f \colon X \to Y$ and an ample divisor $A$ on $X$. 
As we saw in Theorem \ref{BVonMFS}, for the vanishing of $H^1(X, \sO_X(-A))$, it was enough to see a general fiber of $f$, which is normal.
However, as we can see in \cite[Theorem 1.10 and Remark 1.11]{PW17}, for the vanishing of $H^2(X, \sO_X(-A))$, it seems that we need to consider all fibers of $f$, including non-normal fibers.

Let us recall the tameness of a (non-normal) Gorenstein del Pezzo surface.

\begin{defn}
Let $X$ be a (possibly non-normal) projective Gorenstein surface.
We say $X$ is a \emph{Gorenstein del Pezzo surface} if the dualizing sheaf $\omega_X$ is anti-ample.
We say a Gorenstein del Pezzo surface $X$ is \emph{tame} if $\mathcal{X}(\sO_X)=1$.

\end{defn}


\begin{lem}\label{nonnormal}
Let $X$ be a non-normal Gorenstein del Pezzo surface and $\nu\colon X' \to X$ be the normalization.
Let $C\subset X$ and $C'\subset X'$ be the closed subschemes defined by the conductor ideals of $\nu$.
Then the following hold.
\begin{enumerate}\renewcommand{\labelenumi}{$($\textup{\arabic{enumi}}$)$}
\item{$(-K_X)\cdot C=1$.}
\item{$C$ is reduced and irreducible.}
\item{$C\simeq \PP^1_k$ if and only if $X$ is tame.}
\end{enumerate}
\end{lem}

\begin{proof}
(1) By Reid's classification of non-normal Gorenstein del Pezzo surfaces (\cite[1.1 Theorem]{Rei94}, see also \cite[Theorem 5.3]{FS18}), we can check that $\nu^{*}(-K_{X})\cdot C'=2$. 
Then, by the projection formula and \cite[Proposition A.2.]{FS18}, we have $(-K_X)\cdot C=\frac{1}{2}\nu^{*}(-K_{X})\cdot C'=1$.

(2) Since $(-K_X)\cdot C=1$, it follows that $C$ is irreducible and regular in the generic point. Since $X$ satisfies $S_2$, it follows that $C$ and $C'$ satisfy $S_1$ by the proof of \cite[Proposition A.1.(i)]{FS18}. Therefore $C$ is reduced. 

(3) First, we show $\mathcal{X}(\sO_{X'}(-C'))=0$.
Since $C'$ is strictly effective, we have $H^0(X',\sO_{X'}(-C'))=0$.
If $i>0$, then $H^i(X',\sO_{X'}(-C'))=H^{2-i}(X', \sO_{X'}(K_{X'}+C'))=H^{2-i}(X', \sO_{X'}(\nu^{*}K_X))=0$ by Lemma \ref{Cartier Vanishing} since
$X'$ is a normal projective rational surface and $\nu^{*}(-K_X)$ is ample Cartier. Thus we obtain $\mathcal{X}(\sO_{X'}(-C'))=0$.
By the exact sequence 
\[
0\to \sO_{X'}(-C')\to \sO_{X'} \to \sO_{C'}\to 0,
\]
we have $\mathcal{X}(\sO_{X'})=\mathcal{X}(\sO_{C'})+\mathcal{X}(\sO_{X'}(-C'))=\mathcal{X}(\sO_{C'})$.
Now, by \cite[Appendix (30)]{FS18}, we have the following exact sequence
\begin{align}
0 \to \sO_X \to \nu_{*}\sO_{X'}\oplus\sO_C \to \nu_{*}\sO_{C'}\to 0. \tag{5.1}
\end{align}
This shows that $\mathcal{X}(\sO_X)-\mathcal{X}(\sO_C)=\mathcal{X}(\sO_{X'})-\mathcal{X}(\sO_{C'})=0$.
Therefore $C$ is a smooth rational curve if and only if $X$ is tame.

\end{proof}

\begin{lem}\label{tame}
Let $f\colon X\to Y$ be a del Pezzo fibration such that $X$ is smooth and $F$ be a fiber of $f$.
Then $F$ is a tame Gorenstein del Pezzo surface.
\end{lem}
\begin{proof}
Let $F$ be a fiber of $f$. Then $F$ is irreducible since $\rho(X/Y)=1$ and $X$ is smooth. 
Let $G\coloneqq F_{\mathrm{red}}$. Since $K_{G}=(K_X+G)|_{G}\equiv K_{X}|_{G}$ is an anti-ample Cartier divisor, it follows that $G$ is a Gorenstein del Pezzo surface.
We denote $F=mG$ for $m \in \mathbb{Z}_{>0}$.
By the exact sequence $0 \to \sO_G(-G) \to \sO_{2G} \to \sO_{G}\to 0$, we get $\mathcal{X}(\sO_{2G})=\mathcal{X}(\sO_{G})+\mathcal{X}(\sO_{G}(-G))=2\mathcal{X}(\sO_{G})$. 
By repeating this, we obtain $\mathcal{X}(\sO_F)=m\mathcal{X}(\sO_G)$.
Let $F_{\mathrm{gen}}$ be a general fiber of $f$. Since the generic fiber of $f$ is a regular del Pezzo surface, $F_{\mathrm{gen}}$ is normal by \cite[Theorem 14.1]{FS18}. 
Then Lemma \ref{Cartier Vanishing} shows that $H^1(F_{\mathrm{gen}}, \sO_{F_{\mathrm{gen}}})\simeq H^1(F_{\mathrm{gen}}, \sO_{F_{\mathrm{gen}}}(K_{F_{\mathrm{gen}}}))=0$ since $F_{\mathrm{gen}}$ is a normal projective surface with $\kappa(F_{\mathrm{gen}})=-\infty$ and $-K_{F_{\mathrm{gen}}}$ is ample Cartier. 
In particular, $\mathcal{X}(\sO_{F_{\mathrm{gen}}})=1$.
Now, the flatness of $f$ shows that $m\mathcal{X}(\sO_G)=\mathcal{X}(\sO_F)=\mathcal{X}(\sO_{F_{\mathrm{gen}}})=1$. Therefore $m=1$ and $F$ is a tame Gorenstein del Pezzo surface.
\end{proof}

Now, we can prove the Kodaira vanishing theorem for a smooth del Pezzo fibration.

\begin{thm}\label{KVonMFS2}
Let $f \colon X \to Y$ be a del Pezzo fibration and $A$ be an ample Cartier divisor on $X$.
Suppose that $X$ is smooth projective and $p\geq 11$.
Then $H^2(X, \sO_X(-A))=0$.
\end{thm}
\begin{proof}
Let $F$ be a non-normal fiber of $f$.
Then $F$ is a tame Gorenstein del Pezzo surface by Lemma \ref{tame}.
By \cite[Theorem 1.10]{PW17}, 
it suffices to show $H^1(F, \sO_F(-A|_F))=0$.
By Lemma \ref{rho} (2), we can denote $A\equiv a(-K_X)+bF$ for some $a,b\in \Q$. 
By restricting to $F$, we get $A|_F\equiv a(-K_F)$.
Let $\nu \colon F' \to F$ be the normalization and $C\subset F$ and $C'\subset F'$ be the closed subschemes defined by conductor ideals. Then $C\simeq \PP^1_k$ by Lemma \ref{nonnormal} (3).
By the exact sequence (5.1) and $h^1(F, \sO_F)=0$, we obtain $h^0(C', \sO_{C'})=1$.
By \cite[Appendix (30)]{FS18}, we have the exact sequence of multiplicative abelian sheaves
\[
1 \to \sO_F^{\times} \to \nu_{*}\sO_{F'}^{\times}\oplus\sO_C^{\times} \to \nu_{*}\sO_{C'}^{\times}\to 1, 
\]
where the map on the right is $(s,t)\mapsto \frac{s|_{C'}}{\nu^{*}t}$.
Since the restriction map $H^0(F', \sO_{F'}^{\times})=k^{\times} \to H^0(C', \sO_{C'}^{\times})=k^{\times}$ is bijective, it follows that $\Pic(F)\to \Pic(F')\oplus\Pic(C)$ is injective. 
Now since $F'$ and $C$ are a normal rational surface and a smooth rational curve, both of $\Pic(F')$ and $\Pic(C)$ are torsion-free and so is $\Pic(F)$.
By the tameness of $F$, we have $H^1(F, \sO_F)=0$ and thus $\Pic(F)=\NS(F)$.
Together with the torsion-freeness of $\Pic(F)$, we obtain $\Pic(F)= N^1(F)$.
Since $(-K_F)\cdot C=1$ by Lemma \ref{nonnormal} (1), we have $a=a(-K_F)\cdot C=A|_{F}\cdot C\in \Z$.  
Then it follows from $\Pic(F)=N^1(F)$ that $A|_{F}$ is linearly equivalent to $-aK_{F}$.
Now, we get $H^1(F, -A|_{F})=H^1(F, aK_F)=0$ by \cite[4.10 Corollary (1)]{Rei94}.
\end{proof}
\section*{\bf Appendix}\label{Appendix}
For the convenience of the reader, we include the proof of Theorem \ref{fano} as an appendix to this paper. 
We emphasize that all the results in this appendix are proved by Shepherd-Barron \cite{SB18}, whom we thank very much.

\medskip

\noindent{\bf Lemma A.1.}(\cite[Lemma 2.4]{SB18})
{\em Let $f \colon Y \to X$ be a finite dominant morphism of degree two from a normal irreducible scheme to a regular irreducible scheme.
Then $f_{*}\sO_Y/\sO_X$ is an invertible sheaf and $Y$ is l.c.i..}
\begin{proof}
First, we show $f_{*}\sO_Y/\sO_X$ is invertible. We may assume $X$ is a spectrum of a regular local ring $(X, x)$.
We use the induction on $\dim X$.
Let us assume $\dim X\leq2$. In this case, $f$ is flat because $Y$ is Cohen-Macaulay, $X$ is regular, and $f$ is finite.
Since $\sO_X$ is a regular local ring, the map $\sO_X \to f_{*}\sO_Y$ splits by \cite[Theorem 2]{Hoc73}.  
Then $f_{*}\sO_Y/\sO_X$ is a direct summand of the free module $f_{*}\sO_Y$ and thus free of rank one. Now, let us assume $\dim X\geq 3$. Let $U\coloneqq X-{x}$ and $i \colon U \hookrightarrow X$ be the natural inclusion map.
By considering the pushforward by $i$ of the following exact sequence
\[
0 \to \sO_U \to (f_{*}\sO_Y)|_U \to (f_{*}\sO_Y/\sO_X)|_U \to 0,
\]
we have
\[
0 \to \sO_X \to f_{*}\sO_Y \to i_{*}((f_{*}\sO_Y/\sO_X)|_U) \to R^1i_*\sO_U=\mathcal{H}^2_x(\sO_X)\underset{X: S_3}{=}0.
\]
Since $(f_{*}\sO_Y/\sO_X)|_U=(f|_{f^{-1}(U)})_{*}\sO_{f^{-1}(U)}/\sO_U$ is invertible by the induction hypothesis, it follows that $i_{*}(f_{*}(\sO_Y/\sO_X)|_U)$ is reflexive of rank one and thus invertible. 
Next, we prove the latter assertion.
By \cite[Theorem 2]{Hoc73}, we have $f_{*}\sO_Y=\sO_X\oplus t\sO_X$ for some $t\in f_{*}\sO_Y$. 
Then there exists a natural surjective morphism of $\sO_X$-algebra $\phi\colon \sO_X[t]/(t^2+at+b)\to \sO_X\oplus t\sO_X$ for some $a, b\in \sO_X$.
By tensoring the residue field $k\coloneqq \sO_X/\mathrm{m}_{x}$ of $x$, we obtain
\[
0 \to \mathrm{Ker}(\phi)\otimes_{\sO_X}k \to k[t]/(t^2+at+b) \simeq k\oplus kt \to 0,
\]
where the injectivity of the first map follows from the fact that $\sO_X\oplus t\sO_X$ is a flat $\sO_X$-module.
By Nakayama's lemma, we get $\mathrm{Ker}(\phi)=0$ and thus $f_{*}\sO_Y\simeq \sO_X[t]/(t^2+at+b)$.
Therefore $Y$ is l.c.i..
\end{proof}

\medskip

\noindent{\bf Theorem A.2.}(\cite[Theorem 2.1]{SB18}, cf.~\cite[Theorem 1.4]{SB97})
{\em
Let $X$ be a smooth Fano threefold of $\rho(X)=1$ and $D$ be an ample divisor on $X$.
Then $H^1(X, \sO_X(-D))=0$.}
\begin{rem}\label{fanorem}
For the proof, we have to distinguish from the linear equivalence to the numerical equivalence because this is a consequence of Theorem A.2. as in Corollary \ref{app} (2).
Furthermore, we have not known $\mathcal{X}(\sO_X)=1$ because this is also a consequence of Theorem A.2. as in Corollary \ref{app} (1). However,  it follows from the rational chain connectedness of $X$ that $\mathcal{X}(\sO_X)>0$ as we saw in the proof of Theorem \ref{BVonfano}.
\end{rem}
\begin{proof}
We fix an ample Cartier divisor $H$ whose image in $N^1(X)\simeq \Z$ is a generator.
By the Fujita vanishing theorem (\cite[Theorem 1.4.35 and Remark 1.4.36]{Laz}), there exists $n_{0}$ such that $H^2(X, \sO_X(K_X+nH+F))=0$ for every $n\geq n_{0}$ and every nef Cartier divisor $F$ on $X$. Then, for a Cartier divisor $D'$ satisfying $D'\equiv n'H$ for some $n'\geq n_{0}$, we have 
\[
H^1(X, \sO_X(-D'))\simeq 
H^2(X, \sO_X(K_X+D'))=H^2(X, \sO_X(K_X+n'H+(D'-n'H)))=0
\]
since $D'-n'H\equiv0$. 
Now, for the sake of contradiction, let us assume that the statement of the theorem fails.
Then we can take the maximum positive integer $n\in \Z_{>0}$ such that there exists an ample Cartier divisor $D$ satisfying $D\equiv nH$ and $H^1(X, \sO_X(-D))\neq 0$.
We take such an ample Cartier divisor $D$.
Since $H^1(X, \sO_X(-pD))=0$ by the maximality of $n$,
there exists an l.c.i. projective variety $Y$ and a purely inseparable finite morphism $\rho \colon Y \to X$ of degree $p$ such that $-K_{Y}=\rho^*(-K_X+(p-1)D)$ by \cite[Proposition 1.1.2]{Mad}. 
We denote $-K_X\equiv mH$ for $m\in \Z_{>0}$. Then $-K_Y\equiv (m+(p-1)n)\rho^{*}H$.
By \cite[II 5.14 Theorem and 5.15 Remark]{Kol96}, there exists a rational curve $R$ on $Y$ such that 
\[
m+(p-1)n\leq (m+(p-1)n)(\rho^{*}H \cdot R)=(-K_{Y})\cdot R\leq \dim Y+1=4.
\]
Therefore, we get $p=2$ or $3$.

First, we discuss the case where $p=3$. In this case, $n=1$ and $-K_Y\equiv (m+2)\rho^{*}H$.
We first assume $Y$ is normal.
By the construction of $Y$ (see \cite[Proposition 1.1.2]{Mad}), we have 
\begin{align}
\mathcal{X}(\sO_Y)=\mathcal{X}(\rho_{*}\sO_Y)=\mathcal{X}(\sO_X)+\mathcal{X}(\sO_X(D))+\mathcal{X}(\sO_X(2D)). \tag{A.1}
\end{align}

Since $\mathcal{X}(\sO_X)>0$ by Remark \ref{fanorem} and $\rho(X)=1$, an application of the Riemann--Roch theorem gives
\[
\begin{array}{rl}
\mathcal{X}(\sO_X(lD))\underset{(3.3)}{=}&\frac{1}{12}lD\cdot (lD-K_X)\cdot (2lD-K_X)+\frac{1}{12}lD \cdot c_2(X)+\mathcal{X}(\sO_X)\\
=&\frac{1}{12}lD\cdot (lD-K_X)\cdot (2lD-K_X)+\frac{1}{12}lnH \cdot c_2(X)+\mathcal{X}(\sO_X)\\
                                 \underset{(3.4)}{=}&\frac{1}{12}lD\cdot (lD-K_X)\cdot (2lD-K_X)+(\frac{2ln}{m}+1)\mathcal{X}(\sO_X)\\
                                 >&0
\end{array}
\]
for all $l\geq0$.
Together with (A.1), we obtain $h^1(Y, \sO_Y(K_Y))=h^2(Y, \sO_Y)>0$.
By the Serre vanishing theorem, we can take $r\in \Z_{\geq0}$ such that $H^1(Y, \sO_Y(p^rK_Y))\neq0$ and $H^1(Y, \sO_Y(p^{r+1}K_Y))=0$.
Since $Y$ is normal l.c.i., by \cite[Proposition 1.1.2]{Mad} again, we have an l.c.i. projective variety $Z$ and a purely inseparable finite morphism $\sigma \colon Z \to Y$ of degree $p$ such that 
\[
\begin{array}{rl}
-K_{Z}=&\sigma^*(-K_Y+(p-1)(-p^rK_Y))\\
=&(1+(p-1)p^r)\sigma^{*}(-K_Y)\\
   \equiv& (1+2\cdot3^r)(m+2)\sigma^{*}\rho^{*}H.
\end{array}
\]
As before, by \cite[II 5.14 Theorem and 5.15 Remark]{Kol96}, there exists a rational curve $R'$ on $Z$ such that 
\[
9\leq(1+2\cdot3^r)(m+2)\leq(1+2\cdot3^r)(m+2)(\sigma^{*}\rho^{*}H\cdot R')=(-K_Z)\cdot R'\leq 4,
\]
a contradiction.
Therefore we may assume $Y$ is non-normal.
Let $\nu \colon \tilde{Y} \to Y$ be the normalization. We denote $K_{\tilde{Y}}=\nu^{*}K_{Y}-C$, where $C$ is the closed subscheme defined by the conductor ideal of $\nu$.
Since $\tilde{\rho}=\nu \circ \rho \colon \tilde{Y} \to X$ is a purely inseparable morphism of normal varieties of degree $p$, it follows that $\tilde{\rho}$ factors through the Frobenius morphisms of $X$ and $\tilde{Y}$. Therefore, $\tilde{\rho}$ is homeomorphic, $\rho(\tilde{Y})=1$, and $\tilde{Y}$ is $\Q$-factorial by \cite[Lemma 2.5]{Tan}.
Then $C$ is a $\Q$-Cartier Weil divisor and written as $C\equiv \beta\tilde{\rho}^{*}H$ for some $\beta\in \Q_{>0}$. Note that $\beta\neq0$ since $Y$ is non-normal.
\begin{cl}
$H$ is a prime divisor.
\end{cl}
\begin{proof}
By the Riemann--Roch theorem,
\[
\begin{array}{rl}
\mathcal{X}(\sO_X(H))\underset{(3.3)}{=}&\frac{1}{12}H\cdot (H-K_X)\cdot (2H-K_X)+\frac{1}{12}H \cdot c_2(X)+\mathcal{X}(\sO_X)\\
                                 \underset{(3.4)}{=}&\frac{1}{12}H\cdot (H-K_X)\cdot (2H-K_X)+(\frac{2}{m}+1)\mathcal{X}(\sO_X)\\
            >&0.
\end{array}
\]
Since $H-K_X\equiv (m+1)H$, we have $H^2(X, \sO_X(H))\simeq H^1(X, \sO_X(-(H-K_X)))=0$ by $m+1>1=n$ and the definition of $n$.
Thus we get 
\[
h^0(X, \sO_X(H))\geq\mathcal{X}(\sO_X(H))>0
\]
and $H$ is linearly equivalent to an effective Cartier divisor.
Since $H\in N^1(X)$ is a generator, $H$ is a prime divisor.
\end{proof}

Let $T\coloneqq (\tilde{\rho}^{*}H)_{\mathrm{red}}$. Since $\tilde{\rho}$ is homeomorphic, $T$ is a prime divisor and we can write as $\tilde{\rho}^{*}H=rT$ for some $r\in \Z_{>0}$.
Then
\[
3H=\tilde{\rho}_{*}\tilde{\rho}^{*}H=r\tilde{\rho}_{*}T= rsH
\]
for some $s\in \Z_{>0}$. 
Thus $r=3^{a}$ and $T=3^{-a}\tilde{\rho}^{*}H$, where $a\in\{0, 1\}$.
Let $ T' \to T$ be the normalization and $\pi \colon \tilde{T} \to T'$ be the minimal resolution.
By the adjunction formula (\cite[4.1 (4.2.9) and Proposition 4.5]{Kol13}), there exists an effective $\Q$-divisor $\mathrm{Diff}(0)$ such that
\[
\begin{array}{rl}
K_{T'}+\mathrm{Diff}(0)\equiv& (K_{\tilde{Y}}+T)|_{T'}\\
                       =&(\nu^{*}K_{Y}-C+T)|_{T'}\\
                       \equiv&(-(m+2)-\beta+3^{-a})\tilde{\rho}^{*}H|_{T'}\\
                       =&-(m+2+\beta-3^{-a})A,
\end{array}
\]
where $A$ denotes an ample Cartier divisor $\tilde{\rho}^{*}H|_{T'}$.
By pulling back by $\pi$, we get
\[
K_{\tilde{T}}+\Delta_{\tilde{T}}\equiv \pi^{*}(K_{T'}+\mathrm{Diff}(0)) \equiv -(m+2+\beta-3^{-a})\pi^{*}A
\]
for some effective $\Q$-divisor $\Delta_{\tilde{T}}$. Here, we use the minimality of $\pi$ and the fact that the Mumford pullback (see \cite[14.24.]{Bad} for the definition) of an effective divisor is effective.
Since $m+2+\beta-3^{-a}>0$ and $\Delta_{\tilde{T}}$ is effective, $-K_{\tilde{T}}$ is big.
First, we assume $\tilde{T}$ is not isomorphic to $\PP^2_k$.
In this case, $\tilde{T}$ is a blowing up of a minimal ruled surface and has a fibration structure to a smooth curve. Let $F$ be a general fiber of the fibration.
Then $(-K_{\tilde{T}})\cdot F=2$, and $\pi^{*}A\cdot F\geq1$ because $\pi^{*}A$ is a nef and big Cartier divisor. 
Since $\Delta_{\tilde{T}}$ is $\Q$-effective and $F$ is nef, it follows that $\Delta_{\tilde{T}} \cdot F\geq0$.
Together with $\beta>0$ and $a\in\{0,1\}$, 
we obtain
\[
2=(-K_{\tilde{T}})\cdot F
  =(m+2+\beta-3^{-a})(\pi^{*}A\cdot F)+\Delta_{\tilde{T}} \cdot F
  \geq m+2+\beta-3^{-a}
  >2,
\]
a contradiction.
Next, we discuss the case where $\tilde{T} \simeq \PP^2_k$.
In this case, $T'=\tilde{T}\simeq \PP^2_k$.
If $L$ is a line on $T'$, then $(-K_{T'})\cdot L=3$, $A\cdot L\geq1$, and $\mathrm{Diff}(0) \cdot L\geq0$.
Then we obtain
\[
3=(-K_{T'})\cdot L=(m+2+\beta-3^{-a})A\cdot L+\mathrm{Diff}(0) \cdot L\geq m+2+\beta-3^{-a}.
\]
Thus $m=1$ and $A=\sO_{\PP^2_k}(1)$.
Now we have
\[
1=A^2=(\tilde{\rho}^{*}H)^2 \cdot T=3^{-a}(\tilde{\rho}^{*}H)^3=3^{1-a}H^3
            \underset{m=1}{=}3^{1-a}(-K_X)^3.
\]

On the other hand, by the Riemann-Roch theorem, we have
\[
\begin{array}{rl}
\mathcal{X}(\sO_X(-K_X))&\underset{(3.3)}{=}\frac{1}{12}(-K_X)\cdot (-2K_X)\cdot (-3K_X)+\frac{1}{12}(-K_X) \cdot c_2(X)+\mathcal{X}(\sO_X)\\
         &\underset{(3.4)}{=}\frac{1}{2}(-K_X)^3+3\mathcal{X}(\sO_X).\\
\end{array}
\]
In particular, we get $(-K_X)^3\in 2\Z_{>0}$, a contradiction with $(-K_X)^3=3^{a-1}$.

Next, we consider the case where $p=2$.
In this case, $-K_{Y}\equiv(m+n)\rho^*H$.
Let $\nu \colon \tilde{Y} \to Y$ be the normalization. 
We denote $K_{\tilde{Y}}=\nu^{*}K_{Y}-C$, where $C$ is the closed subscheme defined by the conductor ideal of $\nu$.
Then $C$ is a $\Q$-Cartier Weil divisor and written as $C\equiv \beta\tilde{\rho}^{*}H$ for some $\beta\in \Q_{\geq0}$ by the same argument as in the case where $p=3$.
By Lemma A.1., we have the following commutative diagram
\[
\xymatrix{
0 \ar[r] & \sO_X \ar[r]  \ar@{=}[d] & \rho_{*}\sO_Y \ar[r] \ar@{^{(}->}[d] \ar[d] & \sO_X(D) \ar[r] \ar@{^{(}->}[d]& 0\\
0 \ar[r] &  \sO_X  \ar[r] &     \tilde{\rho}_{*}\sO_{\tilde{Y}}                \ar[r] &\sO_X(E) \ar[r] &0\\
}
\]
for some Cartier divisor $E$. Here, it follows from the construction of $\rho\colon Y\to X$ that $\rho_{*}\sO_Y/\sO_X=\sO_X(D)$. Since $\sO_X(E)$ contains an ample invertible sheaf $\sO_X(D)$ and $\rho(X)=1$, $E$ is ample.
Then we have $h^1(\tilde{Y}, \sO_{\tilde{Y}}(K_{\tilde{Y}}))=h^2(\tilde{Y}, \sO_{\tilde{Y}})>0$ by
\[
\mathcal{X}(\sO_{\tilde{Y}})=\mathcal{X}(\tilde{\rho}_{*}\sO_{\tilde{Y}})=\mathcal{X}(\sO_X)+\mathcal{X}(\sO_X(E))>1.
\]
Note that we can use the Serre duality on $\tilde{Y}$ since $\tilde{Y}$ is l.c.i. by Lemma A.1..
By the Serre vanishing theorem, there exists $r\in \Z_{\geq0}$ such that $H^1(\tilde{Y}, \sO_{\tilde{Y}}(p^rK_{\tilde{Y}}))\neq0$ and $H^1(\tilde{Y}, \sO_{\tilde{Y}}(p^{r+1}K_{\tilde{Y}}))=0$.
Then by \cite[Proposition 1.1.2]{Mad}, we have an l.c.i. projective variety $Z$ and a purely inseparable finite map $\sigma \colon Z \to {\tilde{Y}}$ of degree $p$  
such that 
\[
\begin{array}{rl}
-K_{Z}=&\sigma^*(-K_{\tilde{Y}}+(p-1)(-p^rK_{\tilde{Y}}))\\
=&(1+(p-1)p^r)\sigma^*(-K_{\tilde{Y}})\\
=&(1+2^r)\sigma^*(\nu^{*}(-K_{Y})+C)\\
\equiv&(1+2^r)\sigma^*((\nu^*(m+n)\rho^{*}H+\beta\tilde{\rho}^{*}H)\\
=&(1+2^r)(m+n+\beta)\sigma^*\tilde{\rho}^{*}H.\\
\end{array}
\]
Now by \cite[II 5.14 Theorem and 5.15 Remark]{Kol96}, there exists a rational curve $R$ on $Z$ such that $(-K_Z)\cdot R\leq 4$. This shows $m=n=1$ and $r=\beta=0$. In particular, $Y$ is normal and $-K_Z\equiv 4\sigma^*\rho^{*}H$.
Let $\nu' \colon \tilde{Z} \to Z$ be the normalization. 
We denote $\tilde{\tau} \colon \tilde{Z} \overset{\nu'}{\to} Z \overset{\sigma}{\to} Y \overset{\rho}{\to} X$ and $K_{\tilde{Z}}=\nu'^{*}K_Z-C'$, where $C'$ is the closed subscheme defined by the conductor ideal of $\nu'$.
Then $C'$ is a $\Q$-Cartier Weil divisor and written as $C'\equiv \gamma \tilde{\tau}^{*}H$ for some $\gamma\in \Q_{\geq0}$ since $\tilde{Z}$ is $\Q$-factorial and $\rho(\tilde{Z})=1$.
\begin{cl}
$H$ is a prime divisor.
\end{cl}
\begin{proof}
As in the case where $p=3$, it suffices to show $H^1(X, \sO_X(-(H-K_X)))=0$.
Since $H-K_X\equiv (m+1)H=2H$, this follows from $2>1=n$ and the definition of $n$.
\end{proof}

Let $T\coloneqq (\tilde{\tau}^{*}H)_{\mathrm{red}}$. Since $\tilde{\tau}$ is homeomorphic, $T$ is a prime divisor and we can write as $\tilde{\tau}^{*}H=rT$ for some $r\in \Z_{>0}$.
Then
\[
4H=\tilde{\tau}_{*}\tilde{\tau}^{*}H=r\tilde{\tau}_{*}T=rsH
\]
for some $s\in \Z_{>0}$. Thus $r=2^a$ and $T=2^{-a}\tilde{\tau}^{*}H$, where $a\in\{0, 1, 2\}$.
Let $T' \to T$ be the normalization and $\pi \colon \tilde{T} \to T'$ be the minimal resolution.
By the adjunction formula, there exists an effective $\Q$-divisor $\mathrm{Diff}(0)$ such that
\[
\begin{array}{rl}
K_{T'}+\mathrm{Diff}(0)&\equiv (K_{\tilde{Z}}+T)|_{T'}\\
                       &=(\nu'^{*}K_{Z}-C'+T)|_{T'}\\
                       &\equiv(-4-\gamma +2^{-a})\tilde{\tau}^{*}H|_{T'}\\
                       &=-(4+\gamma-2^{-a})A,
\end{array}
\]
where $A$ denotes an ample Cartier divisor $\tilde{\tau}^{*}H|_{T'}$.
Thus we obtain 
\[
K_{\tilde{T}}+\Delta_{\tilde{T}}\equiv \pi^{*}(K_{T'}+\mathrm{Diff}(0))
                                \equiv -(4+\gamma-2^{-a})\pi^{*}A
                                \]
for some effective $\Q$-divisor $\Delta_{\tilde{T}}$.
In particular, $-K_{\tilde{T}}$ is big.
First, we assume $\tilde{T}$ is not isomorphic to $\PP^2_k$.
As in the case where $p=3$, there exists a curve $F$ such that $(-K_{\tilde{T}})\cdot F=2$, $\pi^{*}A\cdot F\geq 1$, and $\Delta_{\tilde{T}} \cdot F\geq0$. 
Together with $\gamma\geq 0$ and $a\in\{0,1,2\}$ we obtain
\[
2=(-K_{\tilde{T}})\cdot F=(4+\gamma-2^{-a})(\pi^{*}A\cdot F)+\Delta_{\tilde{T}} \cdot F\geq4+\gamma-2^{-a}\geq3,
\]
a contradiction.
Next, we discuss the case where $\tilde{T} \simeq \PP^2_k$.
In this case, $T'=\tilde{T}\simeq \PP^2_k$.
If $L$ is a line on $T'$, then $(-K_{T'})\cdot L=3$, $A\cdot L\geq1$, and $\mathrm{Diff}(0) \cdot L\geq0$.
Then we have
\[
3=(-K_{T'})\cdot L=(4+\gamma-2^{-a})A\cdot L+\mathrm{Diff}(0) \cdot L\geq 4+\gamma-2^{-a},
\]
and hence $a=\gamma=0$ and $A=\sO_{\PP^2_k}(1)$.
Now we get
\[
1=A^2=(\tilde{\tau}^{*}H)^2 \cdot T=2^{-a}(\tilde{\tau}^{*}H)^3=2^{2-a}H^3
            \underset{a=0}{=}4H^3\geq4,
\]
a contradiction.

\end{proof}
\begin{acknowledgement}
The author wishes to express his gratitude to his supervisor Professor Shunsuke Takagi for his encouragement, valuable advice, and suggestions. The author is also grateful to Professor Shepherd-Barron for kindly letting him include the proof of \cite{SB18} in an appendix.
He would like to thank the referee for reading the manuscript very carefully, pointing out mistakes, and giving many helpful comments and suggestions.
He is also indebted to Professor Hiromichi Takagi, Professor Hiromu Tanaka, Takeru Fukuoka, Masaru Nagaoka, Kenta Sato, and Shou Yoshikawa for useful discussions and comments. This work was supported by JSPS KAKENHI 19J21085.
\end{acknowledgement}



\end{document}